\documentclass[a4paper,12pt]{article}
\usepackage[T1]{fontenc}
\usepackage{amsfonts,amsthm,amssymb,mathrsfs,amsmath}
\usepackage{graphicx,graphics}
\usepackage{relsize}
\usepackage[all]{xy}
\usepackage{tikz}
\usepackage{color}
\allowdisplaybreaks

\newcommand{\PP}{\mathbb{P}}
\newcommand{\ZZ}{\mathbb{Z}}
\newcommand{\NN}{\mathbb{N}}
\newcommand{\Nc}{\mathcal{N}}
\newcommand{\CC}{\mathbb{C}}
\newcommand{\RR}{\mathbb{R}}
\newcommand{\Hi}{\mathscr{H}}
\newcommand{\Ex}{\mathop{\mathbb{E}}}
\newcommand{\dprod}[2]{\left\langle #1,#2\right\rangle}
\newcommand{\norm}[1]{\left\lVert #1\right\rVert}
\newcommand{\proj}[1]{\left|#1\right\rangle\left\langle#1\right|}

\newcommand{\matx}[1]{\left(\begin{matrix} #1 \end{matrix}\right)}
\newcommand{\clsp}[1]{\overline{\left\langle #1 \right\rangle}}

\newtheorem{theorem}{Theorem}[section]
\newtheorem{lemma}[theorem]{Lemma}

\newtheorem{remark}[theorem]{Remark}
\numberwithin{equation}{section}

\newtheorem{corollary}[theorem]{Corollary}
\newtheorem{hypothesis}[theorem]{Hypothesis}

\begin{document}


\title{Global multiplicity bounds and Spectral Statistics for Random Operators}

\author{
ANISH MALLICK \\ 
Pontificia Universidad Cat\'{o}lica de Chile\\
Rolando Chuaqui Building, San Joaquin Campus. \\
Avda. Vicu\~{n}a Mackenna 4860, Macul, Chile.\\
e-mail: anish.mallick@mat.uc.cl\\
~\\
KRISHNA MADDALY \\ 
Ashoka University \\
Plot No 2, Rajiv Gandhi Education City,\\
Rai 131029, Haryana, India. \\
e-mail: krishna.maddaly@ashoka.edu.in
}


\maketitle


\begin{abstract}
In this paper, we consider Anderson type operators on a separable Hilbert space where the random perturbations are finite rank and the random variables have full support on $\RR$.
We show that spectral multiplicity has a uniform lower bound whenever the lower bound is given on a set of positive Lebesgue measure on the point spectrum away from the continuous one.
We also show a deep connection between the multiplicity of pure point spectrum and local spectral statistics, in particular we show that spectral 
multiplicity higher than one always gives non-Poisson local statistics in the framework of Minami theory.

In particular for higher rank Anderson models with pure-point spectrum, with the randomness having support equal to $\RR$, 
 there is a uniform lower bound on spectral multiplicity and in case this is larger than one the local statistics is not Poisson.
\end{abstract}

{\bf keywords:} Spectral Theory, Random Operators, Perturbation Theory.\\

{\bf MSC-2010:} 81Q10, 47A10, 47A55, 47H40.\\

\section{Introduction}
Random operators are an important field of study for various reasons.
Over the years much focus is given to a certain class of random operator 
like Anderson tight binding model, continuum random Schr\"{o}dinger operator, multi-particle Anderson model and many others.
Some of these models were initially developed to study localization phenomenon and a lot of research is focused on showing the existence of pure point spectrum and exponentially decaying Green's function.

The model considered in this paper is 
\begin{equation}\label{OpDefEq1}
H^\omega=H_0+\sum_{n\in\Nc} \omega_n P_n.
\end{equation}
Typically one takes $H_0$ to be the $-\Delta$ on $L^2(\RR^d)$ with possibly
a vector potential or a periodic background potential added and  
the adjacency operator on $\ell^2(\ZZ^d)$ with $\{\omega_n\}$
 independent random variables with $\{P_n\}_n$ a countable collection of projections.  In
such a setting there are several questions relating to these operators
 that are of interest.  In the mid-fifties
Anderson \cite{PA} proposed that for large disorder the models on
$\ell^2(\ZZ^d)$ should exhibit localization.
Several rigorous results on localization
followed from the early eighties starting with the work of Fr\"{o}hlich-Spencer
\cite{FS} who formulated multi scale analysis. Some of the papers 
on localization for large disorder are
\cite{AENSS, AM, CMS1, CH1, CH2, DLS, FMSS, GK2, GK1, KSS1, SW} and \cite{B1,KN1,DL,HKM}.  
For a comprehensive study of the subject we refer to any of \cite{AW2, CL, FP, K2, PS1} and the references there.

The next set of questions concern the simplicity of the
spectrum and in this direction there are several papers starting from Simon \cite{BS2},
Jak\v{s}i\'{c}-Last \cite{JL1, JL2}, Naboko-Nichols-Stolz \cite{NNS}, Mallick
\cite{AD2, AD1, AM1, AM2, AM3}. From these set of papers, we now know that when the 
rank of $P_n$ is one or for some special cases of higher rank $P_n$,
the singular spectrum is simple.
In the papers \cite{AD1,AN1}, the authors show that there are non-trivial example of random operators where the singular spectrum has non-trivial multiplicity.

Another set of questions of interest are the local spectral statistics
and or  level spacing of the eigenvalues.  The first rigorous  work 
of Molchanov \cite{Mol} led later to the Minami Theory \cite{NM}, which
 establishes a set of sufficient conditions for the local 
spectral statistics to be Poisson.  There are several papers on
local spectral statistics on discrete models such as \cite{AM, AW4, AD3, CGK, DK, GK, FC, KN, MD, Na2, Na1, TV}.
Dietlein and Elgart showed Minami like estimate and thereby showing Poisson local statistics at the spectral edge in case of random Schr\"{o}dinger operators in \cite{DE}.
Their method involves a detailed analysis of the behavior of clusters of eigenvalues possible in the spectrum.

Minami theory \cite{NM} involves looking at the region of complete
exponential localization and prove an inequality now known as
the Minami estimate as part of the proof to show the local 
statistics is Poisson. It was not clear how crucial
the Minami estimate is for determining the local statistics.
Recently Hislop-Krishna \cite{HK}, showed that in all the above models 
exponential localization and Wegner estimate together imply that the
local spectral statistics, whenever it exists, is always Compound Poisson.
The Minami estimate assures us that the L\'evy measure
associated with the limiting infinitely divisible distribution has
support at $\{1\}$ ensuring that the distribution is Poisson.
To state differently, Minami estimate rules out the possibility 
of the limit points, of an array of independent random variables
that are usually constructed in the problem of obtaining local statistics, 
having double points.

The work of Klein-Molchanov \cite{KM2} showed that in the presence of
exponential localization, Minami estimate implies that the point
spectrum is simple when $P_n$ has rank one. The results of the
two papers \cite{KM2} and \cite{HK} raise
an interesting question of what the connection between spectral
multiplicity and the Minami estimate could be.  Our motivation for 
this exposition is to address this question.  
In this paper we consider general unperturbed operators and finite rank $P_n$ 
for the case when the 
distribution of the single site potential has support equal to $\RR$.
We have two very surprising results in this paper, the first is that
{\it the spectral multiplicity of pure point spectrum in any set of positive 
Lebesgue measure gives a lower bound for the spectral multiplicity everywhere in the pure point spectrum.}
The second is that   
{\it Minami estimate gives a sufficient condition for the simplicity
of pure point spectrum even in the case when the rank of $P_n$
 is not one.
}
To our knowledge, these results are new in this generality. 
 This result is not vacuous because Dietlein and Elgart \cite{DE} proved Minami estimate in continuous case, 
though the expression they obtained is similar to expression from Theorem \ref{uniformsingular}. 
Another example is given after the statement of Theorem \ref{uniformsingular}.

The family of random operator we focus on here can be described as follows.
On a separable Hilbert space $\Hi$ the random operators we will focus on are 
of the form given in equation (\ref{OpDefEq1}), satisfying the
following assumptions.

\begin{hypothesis}\label{hyp1}
We assume that $H_0$ is an essentially self-adjoint operator with its 
domain of definition denoted by $\mathcal{D}(H_0)$, 
$\{P_n\}_{n\in\Nc}$ is a countable 
collection of  mutually orthogonal finite rank projections with $$\sum_{n\in\Nc}P_n=I.$$
The real-valued random variables $\{\omega_n\}_{n\in\Nc}$ are mutually 
independent and are distributed with respect to the 
absolutely continuous distribution $g_n(x)dx$, where  
$supp(g_n)=\RR$ for each $n\in\Nc$. 
\end{hypothesis}

In the case of unbounded $H_0$ for example the $-\Delta$
on $L^2(\RR^d)$ with its maximal domain $D(H_0)$, 
taking a countable basis in $D(H_0)$ and considering the finite rank projections generated by them will satisfy the hypothesis.
In case of $-\Delta$, another method is to use smooth basis from $L^2([0,1]^d)$ and translate them using $\ZZ^d$ action to get the desired result.
Anderson tight binding model and Anderson dimer/polymer models fall in this family of operators. 
We will focus on the pure point spectrum away from the continuous spectrum.
Our theorems are valid for several non-ergodic operators, so to 
accommodate those cases, we define
\begin{align*}
\Sigma^\omega=\bigcap_{\substack{B\subset \Nc\\ \# B<\infty}}\bigcap_{n\in B}\left\{E\in\RR:  \lim_{\epsilon\downarrow 0}\norm{\left(H^\omega+\sum_{x\in B} \lambda_x P_x-E-\iota\epsilon\right)^{-1}P_n}<\infty\right.\\
 \text{for almost all $\{\lambda_x\}_{x\in B}$ w.r.t. Lebesgue measure }\bigg\},
\end{align*}
and 
$$\Theta_p=\bigcap_{n\in\Nc}\left\{E\in\RR:  \lim_{\epsilon\downarrow 0}\norm{(H^\omega-E-\iota\epsilon)^{-1}P_n}<\infty~a.s.\right\}.$$
 A discerning reader may note that the condition in the definition of $\Theta_p$ is precisely the  Simon-Wolff criterion 
for pure point spectrum for $H + \lambda P$ a.e. $\lambda$, 
when $P$ has rank one.
For the Anderson model on the lattice, even for higher rank cases, 
the Aizenman-Molchanov method can be used to prove
 localization on $\RR \setminus [-R, R]$ for large enough $R$ under our 
assumptions on $\omega_n$.
So for such models, the set $\Theta_p$ has infinite Lebesgue measure under the Hypothesis \eqref{hyp1}.
Note that in the definition of $\Sigma^\omega$, we are considering perturbations by finitely many projections because if 
$$\lim_{\epsilon\downarrow 0}\norm{\left(H^\omega+\sum_{x\in B} \lambda_x P_x-E-\iota\epsilon\right)^{-1}P_n}<\infty$$
for some $E\in\RR$ and $\{\lambda_x\}_x$, then above equation holds for almost all $\{\lambda_x\}_x$ w.r.t. Lebesgue measure. 

We define  $H_{B}^\omega = P_B H^\omega P_B$, where 
$P_B$ is the orthogonal projection given by $\sum_{n \in B} P_n$, for 
any subset $B \subset {\mathcal N}. $ 
With these definition in place we state a principal result which is the following:
\begin{theorem}\label{mainThm}
Consider the random operators $H^\omega$ on a separable Hilbert 
space $\Hi$, given by equation (\ref{OpDefEq1}), satisfying the 
Hypothesis \eqref{hyp1}.  Assume further that 
$range(P_n)\subset \mathcal{D}(H_0)$ for every $n\in\Nc$.
Suppose there is a subset 
$J \subset \Theta_p$, of positive Lebesgue measure
and a $K < \infty$ such that the multiplicity of eigenvalues of 
$H^\omega$ is bounded below by $K$ on $J$ for a.e. $\omega$. Then, 
\begin{enumerate}
\item Then for every finite $B \subset {\mathcal N}$, the 
multiplicity of any eigenvalue of the  operator
$H_B^\omega: P_B\Hi\rightarrow P_B \Hi$
is bounded below by $K$ for a.e. $\omega$. 
\item The multiplicity of $H^\omega$ on $\Sigma^\omega$ is bounded below by $K$ for almost every $\omega$.
\end{enumerate}
\end{theorem}
\begin{remark}
A few comments are in order before we proceed further.
\begin{enumerate}
\item When $H_0$ is unbounded and $P_n$ satisfy the Hypothesis \ref{hyp1},
it may still happen that the $H^\omega$ are not essentially self-adjoint. 
However there are numerous examples where they are indeed self-adjoint.

\item The hypothesis $supp(g_n)=\RR$ is essential. This is demonstrated by following example:
On the Hilbert space $\ell^2(\ZZ\times\{1,\cdots,5\})$ consider the random operator
$$(H^\omega u)_{n,m}=
\left\{\begin{matrix} 
[u_{n+1,m}+u_{n-1,m}]+\omega_n u_{n,m} & m=1,2 \\ 
2[u_{n+1,m}+u_{n-1,m}]+\omega_n u_{n,m} & m=3,4,5 \end{matrix}
\right.~~~\forall n\in\ZZ, 1\leq m\leq 5,$$
for $u\in \ell^2(\ZZ\times\{1,\cdots,5\})$, where $\{\omega_n\}_{n\in\ZZ}$ are i.i.d random variables following uniform distribution on $[0,1]$.
The projections
$$(Q_i u)_{n,m}=\left\{\begin{matrix}u_{n,i} & m=i\\ 0 & o.w\end{matrix}\right.\qquad \forall (n,m)\in\ZZ\times\{1,\cdots,5\},$$
commute with $H^\omega$ and $Q_iH^\omega Q_i$ is the 
Anderson operator on $\ell^2(\ZZ\times\{i\})$, so it has pure point spectrum.
We notice that $Q_iH^\omega Q_i$ for $i=3,4,5$ are unitarily equivalent, so all the eigenvalues coincide.
Hence the multiplicity of $H^\omega$ is three on the interval $(3,5)$. 
But on the interval $(-2,3)$, the multiplicity is bounded below by $2$, hence the conclusion of theorem \ref{mainThm} fails to hold.
On other hand if we choose $\omega_n$ to be i.i.d random variable following some distribution $g(x)dx$ with $supp(g)=\RR$, 
then using the fact that the spectrum of $Q_1H^\omega Q_1$ is dense in $\RR$, 
we see that the minimum multiplicity of eigenvalues on a given interval is two.
\end{enumerate}
\end{remark}

We define the random variables,  
$$
\eta_{B,J}(\omega) = Tr(E_{H_{B}^\omega}(J)),
$$
for any interval $J$,  where $E_{H_{B}^\omega}$ denotes the spectral 
projection for the operator $H^\omega_B$.

Our main theorem has two remarkable consequences. The first is on the 
multiplicity of the pure point spectrum if Minami estimate \cite{NM},
namely, 
\begin{equation}\label{min-est}
\PP(\{\omega: \eta_{B, J}(\omega) \geq 2\})  \leq C |B|^2 |J|^2,
\end{equation}
holds for any finite $B$ and the constant $C$ independent of 
$B, J$.

\begin{theorem}\label{simplicity}
Consider the operators $H^\omega$ satisfying the Hypothesis \eqref{hyp1}
and let $H^\omega_{\Hi^\omega_B}$ denote $H^\omega$
restricted to the closed $H^\omega$-invariant subspace 
$$\Hi^\omega_B=\overline{\{f(H^\omega)\phi: f\in C_c(\RR),\phi\in P_B\Hi\}}.$$ 
Suppose that $range(P_n)\subset\mathcal{D}(H_0)$ for all $n\in\Nc$ and there is a non-trivial interval 
$I \subset \RR$ and a finite $B \subset \Nc$
 such that
the Minami estimate (\ref{min-est}) holds for 
every subinterval $J \subset I$.  Then the spectrum of 
$H^\omega_{\Hi^\omega_B} $ in $\Sigma^\omega$ is simple.
\end{theorem}

Another extension of Theorem \ref{simplicity} which is  obtained by combining 
Theorem \ref{mainThm} with a result of Anish-Dolai \cite[Lemma 4.1]{AD1} is:

\begin{theorem}\label{uniformsingular}
Let $H^\omega$ satisfy the conditions of Theorem \ref{mainThm}.  
Suppose for a non-trivial interval $I\subset \RR$ and $a,b>0$, the generalized Minami estimate
$$\PP(\omega: \eta_{B,J}(\omega)>K)\leq C|B|^{a}|J|^{1+b}$$
is valid for all $B\subset\Nc$ and any subinterval $J\subset I$ with $a, b >0$ independent of $B, J$.
Then $\Sigma^\omega$ has uniform multiplicity and multiplicity of $\sigma_s(H^\omega)$ is bounded above by $K$.
\end{theorem}

 The generalized Minami estimate stated above usually shows up where $K$ is rank of projection. 
But there are special cases (such as \cite{DE}) where $K<rank(P_n)$. 
As a trivial example, consider $H_0=\sum_{n\in\ZZ} \proj{\delta_{2n}}$ and $P_n=\proj{\delta_{2n}}+\proj{\delta_{2n+1}}$ on the Hilbert space $\ell^2(\ZZ)$.
It is easy to see that the generalized Minami estimate holds (with $K=1$) for any bounded interval with large enough $L$.
The reason generalized Minami estimate is important is because it can show the absence of simple Poisson statistics for the  model.
It was shown in \cite{HK}  
that complete exponential localization and  Wegner estimate are enough
to conclude that limiting statistics in Minami theory \cite{NM} 
is always Compound Poisson.

Suppose for each $N \in \NN$, $B_k, k=1,2, \dots, m_N$, are disjoint 
regions and $I_N$ are intervals such that 
$m_N \rightarrow \infty, ~~ |I_N| \rightarrow 0, N \rightarrow \infty, ~ \cap_N I_N = \{E\}$ 
and consider the array of independent random variables, 
$$
\eta_{k, E, I_N}(\omega) = Tr(E_{H_{B_k}^\omega}(I_N)), ~~ k=1,2,\dots, m_N, ~ N=1, 2, \dots. 
$$

\begin{theorem}\label{CompoundPoisson}
Consider the operators $H^\omega$ as in theorem \ref{mainThm}.
Suppose there exists a set $J \subset \Theta_p$ of positive Lebesgue measure in which the spectrum of $H^\omega$ is pure point and has spectral multiplicity
bigger than one.  
Then for any $E\in\RR$ and any sequence of bounded interval $I_N$ as above,
if the array 
$\{\eta_{k, E, I_N}(\omega), 1 \leq k \leq m_N, N = 1, 2, 3, \dots\}$ 
is asymptotically negligible, 
then the limit points $X_{\omega, E}$ of $\sum_{k}^{m_N} \eta_{k,E,I_N}(\omega)$  are not a Poisson random variables.
\end{theorem}
\subsection*{Ideas of Proofs }
The proof of above theorems are given in next section, however we quickly 
go over the ideas involved in the proofs. 
An important part of the proof of Theorem \ref{mainThm} is to 
study $H+\lambda P$ for a finite rank projection $P$, since
the operator $H^\omega$ can be re-written as \eqref{opDefEq2}. 
Then, the proof of Theorem \ref{mainThm} is divided into four parts, with
Lemma \ref{lem1UniMultRes} and Lemma \ref{lem2StabPPSpec} 
addressing the spectrum of $H+\lambda P$ only and the
Corollary \ref{cor1MultCutoffOp} and \ref{cor2ppSpecMult} use 
the lemmas to conclude the claims of the Theorem \ref{mainThm}.

In Lemma \ref{lem1UniMultRes} we show that if the multiplicity of the 
spectrum of $H+\lambda P$ is bounded below by $K$ in some interval $I$ 
 for almost all $\lambda$, 
then the algebraic multiplicity of eigenvalues of $P(H-z)^{-1}P$ (as a linear operator on $P\Hi$) is bounded below by $K$ for $z\in\CC^{+}$.
We then rewrite  $H^\omega$ as in \eqref{opDefEq2} and use the 
representation \eqref{cor1MultEq1} of $P_B(H^\omega-z)^{-1}P_B$,
to conclude that the multiplicity of the spectrum of $H^\omega_B$ is 
bounded below by $K$. 
This is the idea behind statement (1) of Theorem \ref{mainThm}, the
details of the proof are in the Corollary \ref{cor1MultCutoffOp}.

We then concentrate on the converse, Lemma \ref{lem2StabPPSpec}, namely
if the algebraic multiplicity of roots of the operator $P(H-z)^{-1}P$ (as a linear operator on $P\Hi$) is bounded below by $K$ for $z\in\CC^{+}$,
then the multiplicity of the spectrum of $H+\lambda P$ in 
$$\hat{S}=\{E\in\RR:\lim_{\epsilon\downarrow 0}\norm{(H-E-\iota\epsilon)^{-1}P}<\infty\}$$
is bounded below by $K$.
We then use the representation \eqref{opDefEq2} for 
$H^\omega$ and Lemma \ref{lem1UniMultRes} and Lemma \ref{lem2StabPPSpec} along with 
the fact that $\cup_{B\subset\Nc}\Hi^\omega_B$ is dense subset of $\Hi$, 
to get the lower bound on the multiplicity of spectrum for 
$H^\omega$ in $\Sigma^\omega$.

The main reason for concentrating on the set $\hat{S}$ is 
because in this set, the operator $H_\lambda$ has pure point spectrum, and 
$$dim(ker(H_\lambda-E))=dim(I+\lambda P(H-E-\iota 0)^{-1}P)\qquad a.e.~ E\in\hat{S}$$
with respect to Lebesgue measure, which is shown in Lemma \ref{lem0KerEquiRes}.
Hence any bound on the multiplicity of eigenvalues of $P(H-z)^{-1}P$ translate to a bound on multiplicity of eigenvalues of $H_\lambda$ and vice-versa in the region $\hat{S}$

The proof of Theorem \ref{simplicity} follows all the above steps but 
is more concise owing to the Minami estimate, which guarantees that 
the spectrum of $H^\omega_B$ is simple over the interval $I$ almost surely.
Using the representation \eqref{cor1MultEq1} we get that the matrix $P_B(H^\omega-z)^{-1}P_B$ (as a linear operator on $P_B\Hi$) has simple spectrum for $\Im z\gg 1$.
The discriminant for a polynomial with simple roots is non-zero, so the 
discriminant for the polynomial (in $x$) $$\det(P_B(H^\omega-z)^{-1}P_B-xI),$$
which is the determinant of a Sylvester matrix whose entries are polynomial 
of matrix element of $P_B(H^\omega-z)^{-1}P_B$, is an analytic function on $\CC^{+}$.
So the eigenvalues of the matrix $P_B(H^\omega-z)^{-1}P_B$ are simple for almost all $z\in\CC^{+}$.
So following the steps of proof of Lemma \ref{lem2StabPPSpec} and Corollary \ref{cor2ppSpecMult} completes the proof of the theorem.
Similar approach also works for Theorem \ref{uniformsingular}.

The proof of Theorem \ref{simplicity} is a special case of the technique developed in Lemma \ref{lem1UniMultRes} and Lemma \ref{lem2StabPPSpec}.
But it is easier to follow and provides an insight for the steps involved in the proof of the Lemma \ref{lem1UniMultRes} and Lemma \ref{lem2StabPPSpec}.

\section{Proofs of the Theorems}
In this section we will work with a fixed finite subset $B\subset\Nc$. Let $\{n_i\}_{i=1}^{|B|}$ be an enumeration of $B$ and let $U$ be a real orthogonal matrix of the form
$$U=\matx{ \frac{1}{\sqrt{|B|}} & \frac{1}{\sqrt{|B|}} & \cdots & \frac{1}{\sqrt{|B|}} & \frac{1}{\sqrt{|B|}} \\ u_{2,1} & u_{2,2} & \cdots & u_{2,|B|-1} & u_{2,|B|} \\ \vdots & \vdots & \ddots & \vdots & \vdots\\ u_{|B|,1} & u_{|B|,2} & \cdots & u_{|B|,|B|-1} & u_{|B|,|B|} }.$$
Setting $w_i=e_i^t U\vec{\omega}$, where $\vec{\omega}=(\omega_{n_1},\cdots,\omega_{n_{|B|}})^t$ and $e_i=(0,\cdots,0,1,0,\cdots,0)^t$,
we have 
$$w_1=\frac{1}{\sqrt{|B|}}\sum_{n\in B} \omega_n.$$
For $f_1\in C_c(\RR)$ and $f_2\in C_c(\RR^{|B|-1})$, observe that
\begin{align*}
&\Ex_\omega[f_1(w_1)f_2(w_2,\cdots,w_{|B|})]\\
&=\int f_1\left(\frac{1}{\sqrt{|B|}}\sum_{n\in B} \omega_n\right)f_2\left(\sum_{i=1}^{|B|}u_{2,i}\omega_{n_i},\cdots,\sum_{i=1}^{|B|}u_{|B|,i}\omega_{n_i}\right)\prod_{i=1}^{|B|}g_{n_i}(\omega_{n_i})d\omega_{n_i}\\
&=\int \left(\int f_1(w_1)\prod_{i=1}^{|B|}g_{n_i}\left(\frac{w_1}{\sqrt{|B|g}}+\sum_{j=2}^{|B|} u_{j,i}w_j\right) dw_1\right) f_2(w_2,\cdots,w_{|B|}) \prod_{i=2}^{|B|} dw_i.
\end{align*}
Since $supp(g_{n})=\RR$ we have $\prod_{i=1}^{|B|}g_{n_i}\left(\frac{w_1}{\sqrt{|B|}}+\sum_{j=2}^{|B|} u_{j,i}w_j\right)\neq 0$ for almost all $w_1$ for almost all $w_2,\cdots,w_{|B|}$. 
Hence the conditional distribution of $w_1$ given $w_2,\cdots,w_{|B|}$ is 
absolutely continuous. 
Decomposing the operator $H^\omega$ as
\begin{align}\label{opDefEq2}
H^\omega&=H_0+\sum_{n\in\Nc} \omega_n P_n\nonumber\\
&=H_0+\sum_{n\in\Nc\setminus B}\omega_n P_n+\sum_{j=2}^{|B|} w_j \left(\sum_{i=1}^{|B|} u_{j,i}P_{n_i}\right)+\frac{w_1}{\sqrt{|B|}}\left(\sum_{n\in B}P_n\right),
\end{align}
one can view $w_1$ as a random variable with absolutely continuous distribution depending on $w_2,\cdots,w_{|B|}$.

With above observation, the result boils down to studying the multiplicity problem for single perturbation. 
We only need to work with a fixed essentially self-adjoint operator $H$ on a separable Hilbert space $\Hi$, and set
$$H_\lambda=H+\lambda P$$
for some finite rank projection $P$. Defining the closed $H$-invariant subspace
$$\Hi_P=\clsp{f(H)\phi: f\in C_c(\RR),\phi\in P\Hi},$$
and observe that it is $H_\lambda$-invariant, so using Spectral theorem we have
$$(\Hi_P,H_\lambda)\cong (L^2(\RR,PE_{H_\lambda}(\cdot)P,P\Hi),M_{Id}),$$
where $E_{H_\lambda}$ is the spectral measure associated with the operator $H_\lambda$ and $M_{Id}$ is given by
$$(M_{Id}\psi)(x)=x\psi(x)\qquad \forall x\in\RR$$
for $\psi:\RR\rightarrow P\Hi$ with compact support. Using functional calculus we have
$$P(H_\lambda-z)^{-1}P=\int \frac{1}{x-z} PE_{H_\lambda}(dx)P\qquad \forall z\in\CC\setminus\RR,$$
and using  \cite[Theorem 6.1]{GT1} we can retrieve the measure $PE_{H_\lambda}(\cdot)P$ from $P(H_\lambda-z)^{-1}P$, where we view 
$$P(H_\lambda-z)^{-1}P:P\Hi\rightarrow P\Hi\qquad z\in\CC\setminus\RR,$$
as a linear operator over a finite dimensional vector space $P\Hi$. Denote
$$G^\lambda(z)=P(H_\lambda-z)^{-1}P,~\&~G(z)=P(H-z)^{-1}P\qquad \forall z\in\CC\setminus\RR,$$
and 
$$G^\lambda(E+\iota 0)=\lim_{\epsilon\downarrow 0} G^\lambda(E+\iota \epsilon)~\&~G(E+\iota 0)=\lim_{\epsilon\downarrow 0} G(E+\iota \epsilon)$$
whenever the limit exists for $E\in\RR$. By general theory for matrix valued Herglotz functions  \cite[Theorem 6.1]{GT1}, the set
$$S:=\{E\in\RR: G(E+\iota 0) \text{ exists and finite}\}$$
has full Lebesgue measure. 

\begin{lemma}\label{lem0KerEquiRes}
On a separable Hilbert space $\Hi$ let $H$ be a self-adjoint operator and set $H_\lambda=H+\lambda P$ for $\lambda\in\RR$ and a finite rank projection $P$. Set 
$$\Hi_P=\clsp{f(H)\phi:f\in C_c(\RR), \phi\in P\Hi}$$
to be the $H$-invariant subspace containing the space $P\Hi$ and
\begin{equation*}
\hat{S}=\{E\in S: \lim_{\epsilon\downarrow 0}\norm{(H-E-\iota \epsilon)^{-1}P}\text{ exists and is finite}\}.
\end{equation*}
Then for any $E\in\hat{S}$ the linear map
$$\tilde{P}:ker_{\Hi_P}(H_\lambda-E)\rightarrow ker_{P\Hi}(I+\lambda G(E+\iota 0))$$
defined by $\tilde{P}\phi=P\phi$ is well defined and is a bijection.
\end{lemma}
\begin{proof}
We first note that if $\displaystyle{\lim_{\epsilon \downarrow 0} \|(H - E - i\epsilon)^{-1} P\|} $ is not finite, then there must be 
a vector in the $ker(H - E)$ since the rank of $P$ is finite, since the
limit always exists by monotone convergence theorem.  
Using the observation that $\Hi_P$ is $H_\lambda$-invariant for all $\lambda$, we will view $H$ and $H_\lambda$ as self-adjoint operators on $\Hi_P$ only.
For $E\in\hat{S}$, the map
$$\tilde{P}:ker_{\Hi_P}(H_\lambda-E)\rightarrow ker_{P\Hi}(I+\lambda G(E+\iota 0)),$$
defined by $\tilde{P}\phi=P\phi$ is well defined and linear. 
To see this let $0\neq \phi\in ker_{\Hi_P}(H_\lambda-E)$, then,  
\begin{align}\label{lem0eq1}
&(H+\lambda P-E)\phi=0\nonumber\\
\Rightarrow\qquad & \phi+\lambda (H-E-\iota\epsilon)^{-1}P\phi+\iota\epsilon (H-E-\iota\epsilon)^{-1}\phi=0\qquad\forall \epsilon>0.
\end{align}
But 
$$\lim_{\epsilon\downarrow 0}\iota\epsilon(H-E-\iota\epsilon)^{-1}=-E_{H}(\{E\}),$$
where $E_H$ is the spectral projection of $H$. Therefore
 if $PE_{H}(\{E\})P\neq 0$ then we will have 
$$G(E+\iota\epsilon)=-\frac{1}{\iota\epsilon}PE_H(\{E\})P+\tilde{G}(E+\iota\epsilon),$$
where $\tilde{G}(\cdot)$ is also a matrix-valued Herglotz function. 
(Reason:  By assumption and by polarization consider a vector $\eta$
such that $\langle \eta, ~  P E_{H}(\{E\})P \eta\rangle \neq 0$.  Then denoting
the measure $\mu = \langle \eta, ~  P E_{H}(\{E\})P \eta\rangle$, we have
by Theroem \cite[Theorem 1.3.2(1)]{DKr}, 
$$
\displaystyle{\lim_{\epsilon \rightarrow 0} \epsilon \Im(\langle \eta, ~ G(E+i\epsilon) \eta\rangle) + \mu(\{E\}) = 0. } 
$$
Since the real parts of $G$ and $\tilde{G}$ are the same, the relation follows.
)

This relation shows
that the $\lim_{\epsilon\downarrow 0}G(E+\iota\epsilon)$ does not exist.
So using the fact that $E\in S$, we get
$$\lim_{\epsilon\downarrow 0}\iota\epsilon (H-E-\iota\epsilon)^{-1}\phi=0.$$
Combining the above equation and \eqref{lem0eq1} we have
$$\lim_{\epsilon\downarrow 0} \phi+\lambda (H-E-\iota\epsilon)^{-1}P\phi=0,$$
which implies $P\phi\in ker_{P\Hi}(I+\lambda G(E+\iota 0))$. 
The map is one-one because, if $P\phi=0$ for some $0\neq \phi\in ker_{\Hi_P}(H_\lambda-E)$, then 
\begin{align*}
& (H+\lambda P-E)\phi=0~~\Rightarrow (H-E)\phi=0\\
\Rightarrow~ &\dprod{\psi}{f(H)\phi}=f(E)\dprod{\psi}{\phi}=0\qquad\forall f\in C_c(\RR),\psi\in P\Hi,
\end{align*}
which implies $\phi\perp \Hi_P$, giving a contradiction.

Now for the inverse of the map, define
$$Q:ker_{P\Hi}(I+\lambda G(E+\iota 0))\rightarrow ker_{\Hi_P}(H_\lambda-E)$$
by $Q\phi=(H-E-\iota 0)^{-1}\phi$. By the definition of $\hat{S}$ the element $Q\phi\in\Hi_P$ for any $\phi\in P\Hi$ and $E\in\hat{S}$.
Let $0\neq \psi=Q\phi$ for some $\phi\in ker_{P\Hi}(I+\lambda G(E+\iota 0))$. 
This choice leads us to :
\begin{align*}
\psi: &= Q\phi = (H-E-i0)^{-1} P\phi= (H-E-i\epsilon)^{-1} P\phi + \psi_\epsilon, \\
& ~~where~ \|\psi_\epsilon \| = o(1), \\
 (H_\lambda-E)\psi  &= (H - E - i\epsilon)\psi + \lambda P \psi + i\epsilon \psi \\
&= P\phi + (H - E - i\epsilon)\psi_\epsilon + \lambda P \psi + i\epsilon \psi \\
&= P\phi + (H - E - i\epsilon)\psi_\epsilon + \lambda G(E+i0)\phi + o(1).
\end{align*}
Therefore
\begin{equation}\label{referee}
P(H_\lambda - E)\psi = P \phi + P(H- E -i\epsilon)\psi_\epsilon + \lambda G(E+i0)\phi + o(1).
\end{equation}
Let  $range(P) = \langle f_1, \dots, f_m\rangle$,  then using the fact that  $range(P) \subset D(H_0),$ we have
$$
\langle f_j, (H - E - i\epsilon) \psi_\epsilon\rangle =
\langle (H - E + i\epsilon) f_j,  \psi_\epsilon\rangle = o(1).
$$
Therefore $P(H - E -i\epsilon)^{-1} \psi_\epsilon = o(1)$.  Substituting
equation (\ref{referee}) and using the condition that $\phi \in ker_{P\Hi}(I + \lambda G(E+i0))$, 
$$
P(H_\lambda - E)\psi = P\phi  + \lambda G(E+i0)\phi  = 0,
$$ 
hence $Q\phi\in ker_{\Hi_P}(H_\lambda-E)$. The injectivity of the map $Q$ 
follows from
\begin{align*}
P Q \phi= P(H-E-\iota 0)^{-1}\phi= G(E+\iota 0)\phi=-\frac{1}{\lambda}\phi.
\end{align*}
This completes the proof of lemma because $-\lambda Q$ is the inverse of the map $\hat{P}$.

\end{proof}

 In the following Lemma, we regard $H_\lambda$ as an operator on $\Hi_P$.

\begin{lemma}\label{lem1UniMultRes}
Let $H,H_\lambda$ and $\hat{S}$ be defined as in Lemma \ref{lem0KerEquiRes}, and let $J$ be a 
subset of $\hat{S}$ of positive Lebesgue measure. .
Suppose any eigenvalue of $H_\lambda$ in $J$ has 
multiplicity at least $K\geq 1$ for almost all $\lambda$.
Then all the roots of the polynomial (in $x$)
$$F_z(x)=\det(G(z)-xI),$$
have multiplicity bounded below by $K$, for almost all $z\in\CC^{+}$.
\end{lemma}

\begin{remark}
Note that since $\sigma_c(H)\neq \RR$, the measure $tr(PE_H(\cdot))$ is not equivalent to Lebesgue measure, 
so using F. and M. Riesz theorem \cite[Theorem 2.2]{JL2}, we get that $tr(G(z))\neq 0$ for $z\in\CC^{+}$.

Note that even if $\sigma(H)\cap J=\phi$, the proof of Lemma \ref{lem0KerEquiRes} shows that for
$\lambda$ in $\cup_{E\in J}\{-\frac{1}{x}:x\in\sigma( G(E+\iota 0))\}$
we have 
$$\sigma(H_\lambda)\cap J\neq \phi.$$
So the hypothesis is not vacuous.

\end{remark}

\begin{proof}
Since the subset $J$ is contained inside $\hat{S}$, we can apply Lemma \ref{lem0KerEquiRes} and claim
$$dim(ker_{\Hi_P}(H_\lambda-E))=dim(ker_{P\Hi}(I+\lambda G(E+\iota 0)))\qquad\forall E\in J.$$

Now using above observation for any $E\in J$ and the hypothesis of the lemma, we get that the geometric multiplicity of spectrum for $G(E+\iota 0)$ is at least $K$ for almost all $E$ in $J$. 
In algebraic terms this implies that all the roots of the polynomial (in $x$)
$$F_E(x)=\det(G(E+\iota 0)-xI)$$
has multiplicity bounded below by $K$ for almost all $E$ in $J$. 

Now for $z\in\CC^{+}$, consider the polynomial (let $N=rank(P)$)
$$F_z(x)=\det(G(z)-xI)=\sum_{i=0}^N a_i(z)x^i,$$
and define
$$G_z(x)=gcd\left(F_z(x),\frac{d F_z}{dx}(x)\right)=\sum_{i}p_i(z)x^i,$$
where following Euclid's algorithm we get that $p_i$ are rational polynomials of $\{a_j\}_j$. Now consider 
$$\tilde{F}_z(x)=\frac{F_z(x)}{G_z(x)}=\sum_i q_i(z)x^i,$$
where $q_i$ are rational polynomial of $\{a_j\}_j$ and $\{p_j\}_j$. Since $p_j$ are rational polynomials of $\{a_k\}_k$, we can view $q_i$ to be rational polynomial of $\{a_k\}_k$ only.

First notice that each root of $F_z$ is a root of $\tilde{F}_z$, and each root of $\tilde{F}_z$ has multiplicity one. 
So if all the roots of the polynomial $F_z(x)$ has multiplicity at least $K$, then $(\tilde{F}_z(x))^K$ divides $F_z(x)$ as a polynomial of $x$.
So define the reminder 
$$\mathcal{R}_z(x)=reminder(F_z(x),(\tilde{F}_z(x))^K)=\sum_i r_i(z)x^i,$$
which following division algorithm tells us that $r_i$ are rational polynomial of $\{a_j\}_j$ and $\{q_j\}_j$. 
Since $q_j$ are rational polynomial of $\{a_k\}_k$, we can view $r_i$ as rational polynomial of $\{a_k\}_k$ only. 
We are only interested in numerator of $r_i$ (as a rational polynomial in $\{a_k\}_k$) which will be denoted by $\tilde{r}_i$. 
Note that $\tilde{r}_i$ are defined for $z\in\CC^{+}$ (because they are polynomial of $\{a_k\}_k$ which are defined for $z\in\CC^{+}$).
Now using the fact that all the roots of $F_E(x)$ has multiplicity bounded below by $K$, we have $\tilde{r}_i(E+\iota 0)=0$ for almost all $E\in J$ for all $i$. 
Since $J$ has non-zero Lebesgue measure, using the Privalov Uniqueness
Theorem \cite[page 552]{EM}, we conclude that $r_i\equiv 0, ~~ \forall ~~ i$ for $z\in\CC^{+}$, 
which means $$\mathcal{R}_z(x)=0\qquad  ~~ \forall z\in\CC^{+}.$$
Hence all the roots of $F_z(x)$ have multiplicity bounded below by $K$.

\end{proof}

With the lemma in place, part (1) of Theorem \ref{mainThm} boils down to writing the resolvent $G(z)$ in a certain way and taking the limit $\Im z\rightarrow \infty$. 
This is done in the following corollary to prove the result.

\begin{corollary}\label{cor1MultCutoffOp}
Let $H^\omega$ be defined by \eqref{OpDefEq1} and $J$ satisfies the hypothesis of theorem \ref{mainThm}. 
For any finite subset $B\subset\Nc$, if $range(P_B)\subset\mathcal{D}(H_0)$ then the multiplicity of any eigenvalue of the operator
$$P_B H^\omega P_{B}:P_B\Hi\rightarrow P_B\Hi$$
is bounded below by $K$ almost surely.
\end{corollary}
\begin{proof}
Using the fact that $H^\omega$ can be written as \eqref{opDefEq2}, defining 
$$H^{\omega, \lambda}=H^\omega+\lambda P_B,$$
and by the definition of $\Theta_p$, we have 
$$\norm{(H^\omega-E-\iota 0)^{-1}P_B}<\infty\qquad a.e ~E\in J$$
for almost all $\omega$.
Hence we can use the lemma \ref{lem1UniMultRes} and get that the roots of the polynomial (in $x$)
$$F_z^\omega(x)=\det(P_B (H^\omega-z)^{-1}P_B-xI)$$
has multiplicity bounded below by $K$, for almost all $\omega$ and $z\in\CC^{+}$. 

Now using the resolvent equation for $H^\omega$ and $\tilde{H}^\omega$, where 
$$\tilde{H}^\omega=P_{B}H^\omega P_{B}+(I-P_{B})H^\omega (I-P_{B}),$$
we can write (viewed as an operator on $P_B\Hi$)
\begin{align}\label{cor1MultEq1}
& P_B (H^\omega-z)^{-1}P_B\nonumber\\
&=\left[P_{B}H^\omega P_{B}- zP_{B}-P_{B}H_0(I-P_{B})(\tilde{H}^\omega-z)^{-1}(I-P_{B})H_0P_{B}\right]^{-1}.
\end{align}
(To get this relation we write $H^\omega = \tilde{H}^\omega + K$, and using resolvent equation twice we have
\begin{align*}
(H^\omega - z)^{-1} & = (\tilde{H}^\omega-z)^{-1} + (\tilde{H}^\omega-z)^{-1} K (\tilde{H}^\omega-z)^{-1} \\ 
&\qquad + (H^\omega - z)^{-1}K (\tilde{H}^\omega-z)^{-1} K (\tilde{H}^\omega-z)^{-1} .
\end{align*}
Then take the last term to the left to get
$$
(H^\omega - z)^{-1} \left[ I - K (\tilde{H}^\omega-z)^{-1} K (\tilde{H}^\omega-z)^{-1}\right] = \left[I + (\tilde{H}^\omega-z)^{-1} K\right] (\tilde{H}^\omega-z)^{-1}.
$$
Since action of $\tilde{H}^\omega$ preserves $P_B\Hi$ and $(P_B\Hi)^\perp$ and $K=P_BH_0(I-P_B)+(I-P_B)H_0P_B$, 
once we multiply $P_B$ from left and right in above expression, we are left with
$$G(z)\left[ P_B - P_BK (\tilde{H}^\omega-z)^{-1} K (\tilde{H}^\omega-z)^{-1}P_B\right]=P_B(\tilde{H}^\omega-z)^{-1}P_B.$$
We get the expression \eqref{cor1MultEq1} by using above expression and the fact that $P_B(\tilde{H}^\omega-z)^{-1}P_B$ is same as $(P_BH^\omega P_B-z)^{-1}$.)   

So we conclude that the algebraic multiplicity of eigenvalues of the matrix
$$P_{B}H^\omega P_{B}-P_{B}H_0(I-P_{B})(\tilde{H}^\omega-z)^{-1}(I-P_{B})H_0P_{B}$$
is at least $K$ for almost all $\omega$ and $z\in\CC^{+}$. Using the fact that $(I-P_{B})H_0P_{B}$ and $P_{B}H_0(I-P_{B})$ are finite rank operator hence bounded and 
$$\norm{(\tilde{H}^\omega-z)^{-1}}\leq \frac{1}{\Im z}\qquad \forall z\in\CC^{+},$$
there exists $C_{\omega,B}$ such that
$$\norm{P_{B}H_0(I-P_{B})(\tilde{H}^\omega-z)^{-1}(I-P_{B})H_0P_{B}}<\frac{C_{\omega,B}}{\Im z}.$$
Denoting $D=P_{B}H^\omega P_{B}$ and $C(z)=P_{B}H_0(I-P_{B})(\tilde{H}^\omega-z)^{-1}(I-P_{B})H_0P_{B}$, 
we have the multiplicity of each root of the polynomial (in $x$)
$$\det(D+C(z)-xI)$$
is bounded below by $K$ for almost all $z\in\CC^{+}$. Set
$$\epsilon=\min\{|E_1-E_2|: E_1,E_2\in\sigma(D)~\&~E_1\neq E_2\},$$
then for $\Im z>\frac{3C_{\omega,B}}{\epsilon}$, we have
$$\norm{(D+C(z)-E)\phi}=\norm{C(z)\phi}<\frac{\epsilon}{3}\norm{\phi},$$
where $E\in\sigma(D)$ and $\phi\in P_B \Hi$ be the corresponding eigenvector,  so we conclude that $D+C(z)$ has an eigenvalue in the ball $\{w\in\CC: |w-E|<\epsilon/3\}$. 
On other hand for any eigenvalue $E^z$ of $D+C(z)$ for $\Im z>\frac{3C_{\omega,B}}{\epsilon}$, let $\phi_z$ be the corresponding eigenvector for $E^z$, then
$$\norm{(D-E^z)\phi_z}=\norm{(D+C(z)-E^z)\phi_z-C(z)\phi_z}\leq \frac{\epsilon}{3}\norm{\phi_z},$$
so there is a unique eigenvalue of $D$ in the ball $\{e\in\CC: |e-E^z|<\frac{\epsilon}{3} \}$.

Let $\{E^z_i\}_i$ be an enumeration of the eigenvalues of $D+C(z)$ for $\Im z>\frac{3C_{\omega,B}}{\epsilon}$, then
$$\det(D+C(z)-xI)=\prod_i (E^z_i-x)^{n^z_i}$$
where $n^z_i$ is the algebraic multiplicity of the eigenvalue $E^z_i$. 
Since all the roots of the polynomial $\det(D+C(z)-xI)$ has multiplicity bounded below by $K$, we have $n_i^z\geq K$.
Using the convergence of $E^z_i$ to an eigenvalue of $D$ as $\Im z\rightarrow\infty$ and 
$$\det(D+C(z)-xI)\xrightarrow{\Im z\rightarrow\infty} \det(D-xI)$$
we get that that all the eigenvalues of $D$ have algebraic multiplicity bounded below by $K$. 
Since $D$ is self-adjoint, we have the equality between algebraic and geometric multiplicity, hence proving the corollary.

\end{proof}

For the second part of theorem \ref{mainThm}, we first need to obtain the claim for $H^\omega$-invariant subspaces $\Hi^\omega_P$.
This is done in the following lemma.
\begin{lemma}\label{lem2StabPPSpec}
Let $H,H_\lambda$ be defined as in Lemma \ref{lem0KerEquiRes}.
Suppose that the roots of the polynomial (in $x$)
$$\det(G(z)-xI)$$
have multiplicity at least $K$ for almost all $z\in\CC^{+}$. 
Then the multiplicity of all the eigenvalues in $\hat{S}$ of $H_\lambda$ restricted to $\Hi_P$ is at least $K$ for almost all $\lambda$ with respect to Lebesgue measure.
\end{lemma}
\begin{proof}
Following the steps from the proof of Lemma \ref{lem1UniMultRes}, we can construct the function $\mathcal{R}_z(x)$ from the polynomial (in $x$)
$$F_z(x)=\det(G(z)-xI)$$
and conclude that the algebraic multiplicity of the eigenvalues for $G(E+\iota 0)$ is bounded below by $K$ for almost all $E\in\RR$ with respect to Lebesgue measure. 
Hence the set
\begin{align*}
\tilde{S}=\{E\in S: \text{algebraic multiplicity of any eigenvalue of}\\
G(E+\iota 0) \text{ is bounded below by }K\}
\end{align*}
satisfies $Leb( S\setminus \tilde{S}) = 0.$

Next using Lemma \ref{lem0KerEquiRes} we have
$$dim(ker_{\Hi_P}(H_\lambda-E))=dim(ker_{P\Hi}(I+\lambda G(E+\iota 0)))\qquad\forall E\in \hat{S}.$$
Using the equation 
\begin{align*}
\dprod{\psi}{\Im P(H-E-\iota\epsilon)^{-1}P\psi}&= \epsilon \norm{(H-E-\iota\epsilon)^{-1}\psi}^2\qquad\forall \psi\in P\Hi,
\end{align*}
and the fact that $\Im G(E+\iota0)\geq 0$, we conclude that 
$$\Im G(E+\iota 0)=0\qquad\forall E\in\hat{S},$$
so $G(E+\iota 0)$ is a self-adjoint matrix, hence the geometric and 
algebraic multiplicity of its eigenvalues coincide.
So the lemma is concluded by the fact that on the $\hat{S}\cap \tilde{S}$ the multiplicity of an eigenvalues of $H_\lambda$ is at least $K$ and using Spectral Averaging, (whereby if $A, B$ are self-adjoint operators on $\Hi$
and $M_\lambda$ the operator of multiplication by $\lambda$ in $L^2(\RR)$, then 
the spectral measures of $A\otimes I + M_\lambda \otimes B$ on $L^2(\RR)\otimes \Hi$,
associated with vectors in the range of $B$ are always absolutely
continuous for positive bounded operators $B$, see for example
Krishna-Stollmann \cite{KrS}, from which
it follows that),  we have
$$PE_{H_\lambda}(\hat{S}\setminus \tilde{S})P=0~~a.e~\lambda,$$
so the multiplicity eigenvalues of $H_\lambda$ in $\hat{S}$ is bounded below by $K$ for almost all $\lambda$.

\end{proof}

Now the proof of second part of Theorem \ref{mainThm} is a consequence of lemma \ref{lem1UniMultRes} and \ref{lem2StabPPSpec} along with a density argument.

\begin{corollary}\label{cor2ppSpecMult}
Let $H^\omega$ be defined as \eqref{OpDefEq1} and assume that it satisfies 
the hypothesis of the Theorem \ref{mainThm} on the set $J\subset \Theta_p$ with the
lower bound on the multiplicity given by $K$. 
Then the multiplicity of any eigenvalue in $\Sigma^\omega$ for the operator $H^\omega$ is bounded below by $K$ for almost 
all $\omega$.
\end{corollary}

\begin{proof}
The proof is done for an increasing family of $H^\omega$-invariant Hilbert 
subspaces, the theorem then follows by a density argument.
Let $\{n_i\}_{i\in\NN}$ be an enumeration of $\Nc$ and define
$$H^{\omega, \lambda}_N=H^\omega+\sum_{i=1}^N \lambda_i P_{n_i},$$
and set $Q_N=\sum_{i=1}^N P_{n_i}$. Denote 
$$\Hi^{\omega}_N=\clsp{f(H^\omega)\phi:f\in C_c(\RR),\phi\in Q_N\Hi},$$
and let $Q^\omega_N:\Hi\rightarrow\Hi^\omega_N$ be the canonical projection 
onto $\Hi^\omega_N$.  For any $\psi\in Q_N\Hi$ we have
\begin{align*}
\dprod{\phi}{e^{\iota t H^{\omega, \lambda}_N}\psi}&=\dprod{\phi}{e^{\iota t H^\omega}\psi}+\iota \sum_{i=1}^N \lambda_i\int_0^t \dprod{\phi}{e^{\iota (t-s)H^\omega} P_{n_i} e^{\iota s H^{\omega,\lambda}_N}\psi }ds\\
&=\dprod{e^{-\iota t H^\omega}\phi}{\psi}+\iota \sum_{i=1}^N \lambda_i\int_0^t \dprod{e^{-\iota (t-s)H^\omega}\phi}{ P_{n_i} e^{\iota s H^{\omega,\lambda}_N}\psi }ds\\
&=0\qquad\forall t\in\RR
\end{align*}
for $\phi\in (\Hi^\omega_N)^\perp$, i.e $\Hi^\omega_N$ is  also $H^{\omega, \lambda}_N$-invariant. 
Following the decomposition of \eqref{opDefEq2} we have a change of variables from $\lambda_1,\cdots,\lambda_N$ to $\eta_1,\cdots,\eta_N$ using an orthogonal matrix such that 
$$H^{\omega, \lambda}_N=H^\omega+\sum_{i=2}^N \eta_i\left(\sum_{j=1}^N u_{j,i} P_{n_i}\right)+\frac{\eta_1}{\sqrt{N}} Q_N$$
where $\eta_1=\frac{1}{\sqrt{N}}\sum_{i=1}^N \lambda_i$. So writing 
$$H^{\omega,\eta}_N=H^\omega+\sum_{i=2}^N \eta_i\left(\sum_{j=1}^N u_{j,i} P_{n_i}\right),$$
we have $H^{\omega, \lambda}_N=H^{\omega,\eta}_N+\frac{\eta_1}{\sqrt{N}} Q_N$. To distinguish the variable $\eta_1$ denote 
$$H^{\omega,\eta,\kappa}_N=H^{\omega,\eta}_N+\kappa Q_N.$$

Since $J\subset \Theta_p$, the definition of $\Theta_p$ implies 
$$\norm{(H^{\omega,\eta,\kappa}_N-E-\iota 0)^{-1}Q_N}<\infty\qquad a.e~E\in J,$$
for almost all $\omega,\eta,\kappa$. Hence using the lemma \ref{lem1UniMultRes} for $H^{\omega,\eta,\kappa}_N$,
we conclude that all the roots of the polynomial (in $x$)
$$\det(Q_N(H^{\omega,\eta}_N-z)^{-1}Q_N-xI)$$
have multiplicity at least $K$, where $Q_N(H^{\omega,\eta}_N-z)^{-1}Q_N$ is viewed as a linear operator on $Q_N\Hi$. 
With this observation, hypothesis of lemma \ref{lem2StabPPSpec} is satisfied.
So we conclude that the multiplicity of spectrum of $H^{\omega,\eta,\kappa}_N$ restricted to $\Hi^\omega_N$ (which is $H^{\omega,\eta,\kappa}_N$-invariant) in 
\begin{align*}
\hat{S}^{\omega,\eta}=\{E\in\RR: Q_N(H^{\omega,\eta}_N-E-\iota 0)^{-1}Q_N\text{ exists and finite, and}\\
\norm{(H^{\omega,\eta}_N-E-\iota 0)^{-1}Q_N}<\infty\}
\end{align*}
is bounded below by $K$, for almost all $\kappa$. Since the second condition 
on the set $\hat{S}^{\omega,\eta}$ is precisely the Lemma \ref{lem0KerEquiRes}
criterion for pure point spectrum, we have 
$$Leb(\Sigma^{\omega,\eta}\setminus \hat{S}^{\omega,\eta})=0,$$
where $\Sigma^{\omega,\eta}$ is same as $\Sigma^{\tilde{\omega}}$ where $\tilde{\omega}$ is such that $H^{\tilde{\omega}}=H^{\omega,\eta}_N$.

Hence we conclude that for almost all $\lambda$, the multiplicity of the operator $H^{\omega,\lambda}_N$ restricted on the invariant subspace $\Hi^\omega_N$ is bounded below by $K$.
This also implies that the multiplicity of the operator $H^{\omega}$ restricted onto the invariant subspace $\Hi^\omega_N$ is bounded below by $K$ for almost all $\omega$. 
This follows because $\{\omega_{n_i}\}_{i=1}^N$ are independent of $\{\omega_n\}_{n\in\Nc\setminus\{n_i:1\leq i\leq N\}}$.
Now using the inclusion $\Hi^{\omega}_N\subseteq \Hi^{\omega}_{N+1}$ for all $N$ which is implied by our hypothesis \ref{hyp1} on $P_n$,, the subspace
$$\tilde{\Hi}^{\omega}:=\bigcup_{N\in\NN} \Hi^{\omega}_N$$
is $H^\omega$-invariant subspace of $\Hi$.
By above argument it is clear that the multiplicity of spectrum, 
in $\Sigma^\omega$ for $H^\omega$ restricted on closure of $\tilde{\Hi}^{\omega}$, is bounded below by $K$. 
We get the conclusion of the corollary by observing that 
$\tilde{\Hi}^{\omega}$ is dense in $\Hi$ because of $Q_N\rightarrow I$ 
strongly.

\end{proof}  

\leftline{\bf Proof of Theorem \ref{mainThm} :}
By hypothesis of the theorem, the hypothesis of Corollary \ref{cor1MultCutoffOp} is satisfied hence part (1) of the theorem is proved.
For the second part, Corollary \ref{cor2ppSpecMult} gives the proof of the statement.

\qed

\leftline{\bf Proof of Theorem \ref{simplicity}:} 
By the definition of $\Sigma^\omega$,  
our Lemma \ref{lem0KerEquiRes} and the comment at the beginning of
the proof of Lemma \ref{lem0KerEquiRes} together with    
the $H^\omega$ invariance of $\Hi^\omega_B$ imply that 
$H^\omega_{\Hi^\omega_B}$ has no continuous component of spectrum in 
$\Sigma^\omega$. 
We then start with a proof of simplicity of the spectrum of 
$H^\omega_B$ in  $I$.  To this end take $I = [a, b]$ and  set 
$$I_{N,n}=\left[a+\frac{b-a}{N}n,a+\frac{b-a}{N}(n+2)\right],  ~~ n\in\{0,\cdots,N-2\}, ~~ N \in \NN. 
$$
Then, using the Minami estimate we have for each $N \in \NN$,
\begin{align*}
&\PP(\{\omega: \exists E\in I\cap\sigma(H^\omega_B)~\text{ such that $E$ has multiplicity higher than one}\})\\
&\qquad\leq \PP(\left\{ \omega :\eta_{B,I_{N,n}}(\omega)\geq 2 ~~ \mathrm{for ~ some} ~~  n\in\{0,\cdots,N-2\}\right\})\\
&\qquad\leq \sum_{n=0}^{N-2}\PP(\eta_{B,I_{N,n}})\leq  \frac{4C|B|^2|I|^2}{N},
\end{align*}
which converges to zero as $N \rightarrow \infty$.

Following the steps involved in obtaining \eqref{opDefEq2}, we can write
$$H^\omega_B=H^{\tilde{\omega}}_B+w P_B,$$
where $w=\frac{1}{|B|}\sum_{n\in B}\omega_n$ and $w$ is real random 
variable with an absolutely continuous distribution, depending on 
$\tilde{\omega}$, having non-zero density at all points.

Since $H^\omega_B$ is an operator on $P_B\Hi$ and $P_B$ acts as the 
 identity operator on $P_B\Hi$, we have 
$\sigma(H^{\tilde{\omega}}_B+w P_B)=w+\sigma(H^{\tilde{\omega}}_B).$
Combining these two facts we see that any eigenvalues of 
$\sigma(H^{\tilde{\omega}}_B)\cap (I-w)$ is almost surely simple for 
almost all $w$. Since $\cup_{w\in\RR}I-w=\RR$, we conclude that 
$\sigma(H^{\tilde{\omega}}_B)$ has simple spectrum.
Since $H^\omega_B-H^{\tilde{\omega}}_B$ is a multiple of  identity, 
we conclude that $H^\omega_B$ also has simple spectrum a.e. $\omega$,
let us denote this set of full measure to be $\Omega_B$.

It remains to show that the  simplicity of spectrum of 
$H^\omega_B$ implies the simplicity of eigenvalues of  
$P_B(H^\omega-z)^{-1}P_B$, as a linear operator on $P_B\Hi$
for a.e. $\omega$. 

The simplicity of the spectrum of $P_B(H^\omega-z)^{-1}P_B$
follows if we show that the  discriminant $\Delta^\omega(z)$ of the polynomial 
$$F^\omega_z(x)=\det(P_B(H^\omega-z)^{-1}P_B-xI)$$
is non-vanishing.  Now, $\Delta^\omega(z)$ 
can be written as the determinant of the Sylvester matrix of $F^\omega_z$ 
and it's derivative, which is are analytic functions of $z$ in $\CC^{+}$. 

Since $H^\omega_B$ has simple spectrum, using the representation \eqref{cor1MultEq1} for $P_B(H^\omega-z)^{-1}P_B$ we conclude that $\Delta^\omega(z)\neq 0$ for $\Im z\gg 1$, 
which implies that $\Delta^\omega(z)\neq 0$ for almost all $z\in\CC^{+}$.  
Hence, by Privalov uniqueness theorem \cite[page 552]{EM},  
$$\lim_{\epsilon\downarrow0}\Delta^\omega(E+\iota \epsilon)\neq 0\qquad a.a~E\in\RR,$$ 
which gives the simplicity of spectrum of $P_B(H^\omega-E-\iota 0)^{-1}P_B$ for almost all $E$.
So using Lemma \ref{lem0KerEquiRes} we conclude that the operator $H^{\omega,\lambda}_B$ restricted on the invariant subspace $\Hi^\omega_B$ on the set $\Sigma^\omega$ has simple spectrum.


So using the fact that $\Hi^\omega_B$ is $H^{\omega,\lambda}_B$-invariant, 
we get the simplicity of spectrum of $H^{\tilde{\omega}}$ in $\Sigma^{\tilde{\omega}}$ on the subspace $\Hi^{\tilde{\omega}}_B$ for almost all $\tilde{\omega}$.

\qed

\leftline{\bf Proof of Theorem \ref{uniformsingular}:}

By using argument from the proof of theorem \ref{simplicity}, we conclude 
that the multiplicity of the spectrum for $H^\omega_B$ is bounded above by $K$. 
Hence using the decomposition \eqref{cor1MultEq1} of $P_B(H^\omega-z)^{-1}P_B$ and following the argument of the corollary \ref{cor1MultCutoffOp} we conclude that the multiplicity of any roots of the polynomial (in $x$)
$$F_z(x)=\det(P_B(H^\omega-z)^{-1}P_B-xI)$$
is bounded above by $K$ for almost all $z\in\CC^{+}$ almost surely, where $P_B(H^\omega-z)^{-1}P_B$ is viewed as a linear operator on $P_B\Hi$ for any finite $B\subset \Nc$.
So using \cite[Theorem 1.1]{AD1} we conclude that the maximum multiplicity of $H^\omega$ is bounded above by $K$.
This completes the proof of the theorem.

\qed


\noindent{\bf Proof of Theorem \ref{CompoundPoisson}:}

The assumption that the spectral multiplicity in 
$J \subset \RR\setminus\Sigma_c$ is bigger than one implies that the spectral multiplicity of
$\sigma(H^\omega_{B})$ is bigger than one for any finite $B\subset\Nc$, by Theorem \ref{mainThm}. 
Therefore for any finite subset $B \subset {\mathcal N}$, 
 and any interval $I \subset \RR$, we have
$$
\PP(\{ \omega : \eta_{B, \omega}(J) = 1 \}) = 0, 
$$
showing that 
$$
\PP(\{ \omega : \eta_{k,E, I_N}(\omega) = 1 \}) = 0, ~ \forall k \leq m_N, ~~ \forall ~~ N \in \NN. 
$$
Therefore, if $\{\eta_{k, E, I_N}^\omega\}$ is a {\it uniformly asymptotically
negligible array} of random variables, (see \cite[Section 11.2]{VJ},
then Theorem 11.2 of  \cite{VJ}, applied to random variables, shows
that the limit 
$$
X_{\omega,E} = \lim_{N\rightarrow \infty} \sum_{k=1}^{m_N}\eta_{k, E, I_N}^\omega
$$  
with the convergence in distribution, is not a Poisson random variable. 
The  proof of \cite[theorem 5.1]{HK}, also gives an alternative proof of
the Theorem.  

\qed

\section*{Acknowledgments}
We thank Dhriti R. Dolai for discussions and Ashoka University for local hospitality to AM where part of this work was done.

\bibliographystyle{unsrt}

\end{document}